\documentclass[12pt]{article}
\usepackage{fullpage}
\usepackage{setspace}
\usepackage{lscape}
\usepackage{amsmath}
\usepackage{amsthm}
\usepackage{acronym}
\usepackage{graphicx}
\usepackage{amssymb}
\usepackage{bm}
\usepackage{mathtools}
\usepackage{booktabs}
\usepackage{multirow}
\usepackage{hyperref}
\usepackage{setspace}
\usepackage{natbib}
\usepackage{color}
\usepackage{caption}
\usepackage{amsfonts}
\usepackage{bm}
\usepackage{mathrsfs}
\usepackage{mathptmx}
\usepackage{mathrsfs}
\usepackage{subcaption}
\usepackage{latexsym}
\usepackage{nicefrac}
\usepackage{verbatim}
\usepackage{lineno}

\everydisplay{\abovedisplayshortskip\belowdisplayshortskip}
\makeatletter \oddsidemargin  -.1in \evensidemargin -.1in

\newtheorem{Theorem}{Theorem}[section]
\theoremstyle{plain}
\newtheorem{Lemma}{Lemma}[section]
\title{Improved estimators in Bell regression model with application}
\author{Solmaz Seifollahi$^1$, Hossein Bevrani$^2$\footnote{Corresponding author} \space  and Zakariya Yahya Algamal$^3$\\
\small {$^1$ Department of Statistics, University of Tabriz, Tabriz, Iran}\\
\small {Email: s.seifollahi@tabrizu.ac.ir}\\
\small {$^2$ Department of Statistics, University of Tabriz, Tabriz, Iran}\\
\small {Email: bevrani@gmail.com}\\
\small {$^3$  College of Computer Science and Mathematics, University of Mosul, Mosul, Iraq} \\
\small {Email: zakariya.algamal@uomosul.edu.iq}
}

\date{}

\begin{document}

\maketitle

\noindent{\bf {\em Abstract:}}

In this paper, we propose the application of shrinkage strategies to estimate coefficients in the Bell regression models when prior information about the coefficients is available. The Bell regression models are well-suited for modeling count data with multiple covariates.
Furthermore, we provide a detailed explanation of the asymptotic properties of the proposed estimators, including asymptotic biases and mean squared errors. To assess the performance of the estimators, we conduct numerical studies using Monte Carlo simulations and evaluate their simulated relative efficiency.
The results demonstrate that the suggested estimators outperform the unrestricted estimator when prior information is taken into account. Additionally, we present an empirical application to demonstrate the practical utility of the suggested estimators.

\vskip 3mm
\noindent {\bf {\em {\color{blue} MSC2020 subject classifications: 62J12, 62F07, 62F12, 62F30}}}
\vskip 6mm

\section{Introduction}
The Poisson regression model (PRM) is commonly used for response variables in the form of count data. However, the PRM assumes that the mean and variance are equal, which is rarely the practice case. This can lead to errors in inference, especially when the data exhibits over-dispersion. In such cases, an alternative approach is to use Bell regression models (BRMs). BRMs have gained popularity in various domains, including traffic accident analysis, insurance claims data, medical research (such as analyzing the relationship between hospital admissions and beds), and environmental studies (e.g., analyzing the relationship between fish catch and water temperature).

Previous studies on BRMs have primarily focused on estimating the model coefficients using the maximum likelihood method, initially introduced by \cite{Castellares}. Subsequently, researchers have addressed the issue of multicollinearity in BRMs and proposed various estimators, such as the Liu estimator (\cite{Majid}), Liu-type estimator (\cite{Ertan}), Ridge estimator (\cite{Amin}), modified Kibria-Lukman estimator (\cite{Shewaa}) and Jacknifed Kibria-Lukman estimator (\cite{Abduljabbar}).

In regression models where prior information about the coefficients $\boldsymbol{\beta}$ is available, shrinkage strategies, such as the pretest approach (\cite{Bancroft}), Stein estimator (\cite{Stein}), and positive Stein estimators (\cite{Kibria}), can be applied to improve coefficient estimation. This involves considering both a model without restrictions and a restricted model, and combining estimators based on these models to enhance the accuracy of the estimates. For more detailed information, refer to \cite{Ahmed} and \cite{Kibria}.

Shrinkage strategies have been successfully applied in various regression models, including the inverse Gaussian regression model\cite{Akram}, the Poisson regression model by \cite{Amin2}, the Gamma regression model by \cite{Arabi}, the Beta regression model by \cite{Seifollahi}. However, to the best of our knowledge, these shrinkage methods have not been applied to BRMs. Therefore, in this paper, we introduce the James-Stein estimator, positive James-Stein estimator, and pretest estimator to BRMs, aiming to enhance coefficient estimation within this context.

The organization of the paper is as follows:  The BRM and the unrestricted estimator (UN) are introduced in Section \ref{sec2}.
In Section \ref{sec3}, the restricted, pretest, James-Stein, and positive James-Stein estimators in BRMs are introduced. The asymptotic properties of suggested estimators are provided in Section \ref{sec4}. The details of the Monte Carlo simulation experiment to compare the performance of the suggested estimators are provided in Section \ref{sec5}. The suggested estimators are applied to a real data set in Section \ref{sec6}. Finally, conclusive remarks are presented in Section \ref{sec7}.

\section{Bell regression model}\label{sec2}
The Bell regression model is suitable when the response variable follows a Bell distribution with the probability mass function as follows:
\begin{equation}\label{bell:d}
p(Y=y)= \dfrac{\phi^y e^{e^{\phi}+1} B_y}{y!}    \qquad y=0, 1, 2, \ldots
\end{equation}
where $ \phi>0 $ and $B_y=\dfrac{1}{e} \sum_{l=0}^\infty \dfrac{l^n}{l!}$ denotes  the Bell numbers (\cite{Castellares, Bell1, Bell2}).
The Bell distribution in \eqref{bell:d} has the following properties:
\begin{align}
\mathbb{E}(Y)&= \phi e^{\phi}\label{bell:m}\\
\mathbb{V}(Y)&= \phi e^{\phi}(1+ \phi)\label{bell:v}
\end{align}
The above expressions show that the value of the variance of Bell distribution is larger than the mean. Consequently, this model is appropriate for the count data with a variance larger than its mean in contrary to the Poisson distribution.

The probability mass function in \eqref{bell:d} shows that this distribution belongs to the exponential family of distributions and we can state this model belongs to the generalized linear models (GLM). Assume that $\mu=\phi e^{\phi}$, then $\phi=W_0(\mu)$ which stands for the Lambert function. Therefore, the Bell distribution can be re-parameterized as follows:
\begin{equation}\label{ndist}
p(Y=y)= e^{1- e^{W_0(\mu)}}\dfrac{W_0(\mu)^y B_y}{y!}    \qquad y=0, 1, 2, \ldots
\end{equation}
In the context of GLM, the relationship between the mean of the response variable and the covariates is  conditionally established by link function. The linear function is expressed as $\boldsymbol{\eta}_i= \boldsymbol{x}_i^T \boldsymbol{\beta}$ for $i=1, 2, \ldots, n$ where $\boldsymbol{x}_i=(1, x_{i1}, x_{i2}, \ldots, x_{ip})^T$ is the $i$th observation of the covariates and $ \boldsymbol{\beta}=(\beta_0, \beta_1, \beta_2, \ldots, \beta_p)^T $ is the model coefficients which are unknown and should be estimated.
The most important object in the GLMs, is the link function which plays a main key to connect the covariates with the mean of the response variable as $\mu_i= g^{-1}(\boldsymbol{x}_i^T \boldsymbol{\beta})$. So, in the BRM, the link function must map $(0, \infty)$ to $\mathbb{R}$ such as the log link function. Therefore, the BRM can be modeled by assuming
\begin{equation}\label{modbell}
\mu_i= \boldsymbol{x}_i^T \boldsymbol{\beta} e^{\boldsymbol{x}_i^T \boldsymbol{\beta}}, \qquad i=1, 2, \ldots, n
\end{equation}
The coefficients estimation in the BRM is achieved through using the maximum likelihood estimator (MLE) based on the iteratively re-weighted least-squares algorithm.
The log-likelihood of BRM is given by:
\begin{equation}\label{log:lik}
\mathcal{L}(\boldsymbol{\beta})= \sum_{i=1}^n y_i log(e^{\boldsymbol{x}_i^T \boldsymbol{\beta}} e^{ e^{\boldsymbol{x}_i^T\boldsymbol{\beta}}}) +  \sum_{i=1}^n \bigg[1- e^{ e^{\boldsymbol{x}_i^T\boldsymbol{\beta}}e^{e^{\boldsymbol{x}_i^T\boldsymbol{\beta}}}}+ log B_{y_i}\bigg] - log \bigg(\prod_{i=1}^n y_i ! \bigg)
\end{equation}

Then, the MLE is obtained by setting the first derivative of equation \eqref{log:lik} to zero. This derivative cannot be solved analytically because it is nonlinear in $ \boldsymbol{\beta} $.
 To overcome this challenge, the Fisher-scoring algorithm can be employed. In each iteration of this algorithm, the parameter is updated using the following method.
\begin{equation}
\boldsymbol{\beta}^{(r+1)}=\boldsymbol{\beta}^{(r)}+I^{-1} (\boldsymbol{\beta}^{(r)}) S(\boldsymbol{\beta}^{(r)})
\end{equation}
where $ I (\boldsymbol{\beta})= - \mathbb{E}(\dfrac{\partial^2}{\partial \boldsymbol{\beta} \partial \boldsymbol{\beta}^T}  \mathcal{L}(\boldsymbol{\beta})) $ and $ S (\boldsymbol{\beta})= - \mathbb{E}(\dfrac{\partial}{\partial \boldsymbol{\beta}}  \mathcal{L}(\boldsymbol{\beta})) $.  Thus, the estimated coefficients are defined as
\begin{equation}
\hat{\boldsymbol{\beta}}_{MLE}=(\boldsymbol{X}^T \hat{\boldsymbol{V}} \boldsymbol{X})^{-1} \boldsymbol{X}\hat{\boldsymbol{V}} \hat{\boldsymbol{w}}
\end{equation}
where  $ \hat{\boldsymbol{V}}= diag(\hat{v}_1, \hat{v}_2, \ldots, \hat{v}_n)$,  $\hat{\boldsymbol{w}}= (\hat{w}_1, \hat{w}_2, \ldots, \hat{w}_n)^T$, $v_i=\dfrac{e^{2\hat{\eta}_i}}{\hat{\mathbb{V}}_i} $, $\hat{\mathbb{V}}_i= \hat{\mu}_i [1+ W_0(\hat{\mu}_i)]$ and $\hat{w}_i= log(\hat{\mu}_i)+\dfrac{y_i - \hat{\mu}_i}{\sqrt{v_i \hat{\mathbb{V}}_i}}$.
In this paper, we call the MLE, the UN and will donate that with $ \hat{\boldsymbol{\beta}}^{UN}$. Under regularity condition, we know that:
\begin{align}
\sqrt{n}(\hat{\boldsymbol{\beta}}^{UN}- \boldsymbol{\beta}) \sim N(\boldsymbol{0}, F^{-1})
\end{align}
where $\boldsymbol{F}= \boldsymbol{X}^T \hat{\boldsymbol{V}} \boldsymbol{X}$

\section{Suggested Estimators} \label{sec3}
In this section, we first obtain the restricted estimator (RE) based on the prior information about coefficients then we will use shrinkage methods to introduce a new estimator in BRMs.
\subsection{Restricted estimator}
when there exists some prior information about the coefficients of the model as linear restrictions, it should be considered in the model to get advantages of this information regarding the coefficients to increase the accuracy of the estimates. Consider that the prior information in the models can be involved by:
\begin{equation}\label{H0}
H_0 :~ \boldsymbol{H}\boldsymbol{\beta}=\boldsymbol{h} \quad vs \quad H_1 :~ \boldsymbol{H}\boldsymbol{\beta} \neq \boldsymbol{h}
\end{equation}
where $ \boldsymbol{H} $ and $ \boldsymbol{h} $ are specified as a $r \times (p+1)$ matrix and a vector of length $r$, respectively.
Following \cite{Kibria}, the RE in the BRM is defined by:
\begin{equation}\label{RE}
\hat{\boldsymbol{\beta}}^{RE}= \hat{\boldsymbol{\beta}}^{UN}- \boldsymbol{F}^{-1} \boldsymbol{H}^T (\boldsymbol{H} \boldsymbol{F}^{-1}\boldsymbol{H}^T)^{-1} (\boldsymbol{H}\hat{\boldsymbol{\beta}}^{UN} - \boldsymbol{h})
\end{equation}
The estimator in \eqref{RE} is based on the prior information in \eqref{H0}. Therefore before using this information, it would be preferred to check its significance. To do so, we can use the following test statistic:
\begin{align}\label{TestSt}
F_n &= 2 \big[\mathcal{L}(\hat{\boldsymbol{\beta}}^{UN})- \mathcal{L}(\hat{\boldsymbol{\beta}}^{RE})\big]\nonumber \\
&= (\boldsymbol{H}\hat{\boldsymbol{\beta}}^{UN} - \boldsymbol{h})^T (\boldsymbol{H} \boldsymbol{F}^{-1}\boldsymbol{H}^T)^{-1}(\boldsymbol{H}\hat{\boldsymbol{\beta}}^{UN} - \boldsymbol{h})
\end{align}
When sample size, $n$, increases the test statistic will follow a chi-square distribution with $r$ degrees of freedom.

\subsection{Pretest estimator}
The pretest estimator (PTE) has the following form
\begin{equation}\label{PTE}
\hat{\boldsymbol{\beta}}^{PTE}= \hat{\boldsymbol{\beta}}^{UN}+ (\hat{\boldsymbol{\beta}}^{RE}-\hat{\boldsymbol{\beta}}^{UN})I_{(F_n< \chi^2_{(r, \alpha)})}
\end{equation}
where $I(.)$ is the indicator function and $\alpha$ is the significant level for testing \eqref{H0}.
The pretest estimator has two options, when $H_0$ is accepted, $\hat{\boldsymbol{\beta}}^{PTE}= \hat{\boldsymbol{\beta}}^{RE}$ and when $H_0$ is rejected $\hat{\boldsymbol{\beta}}^{PTE}= \hat{\boldsymbol{\beta}}^{UN}$.

\subsection{James-Stein estimator}
This estimator combines the unrestricted and restricted estimators in the optimal way to dominate the UN.
The James-Stein estimator (JSE) is defined as:
\begin{equation}\label{SE}
\hat{\boldsymbol{\beta}}^{JSE}= \hat{\boldsymbol{\beta}}^{RE}+ \big(1- (r-2)F_n^{-1}\big)
(\hat{\boldsymbol{\beta}}^{RE}-\hat{\boldsymbol{\beta}}^{UN}) \qquad r \geq 3
\end{equation}

\subsection{Positive James-Stein estimator}
In the JSE, it is possible that  $1- (r-2)F_n^{-1}<0$ or equivalent to $F_n<r-2$ which causes over-shrinkage in the estimating process. To avoid this issue, we suggest the positive James-Stein estimator (PJSE) in the BRMs.
The PJSE combines  is defined as:
\begin{equation}\label{PSE}
\hat{\boldsymbol{\beta}}^{PJSE}= \hat{\boldsymbol{\beta}}^{RE}+ \big(1- (r-2)F_n^{-1}\big)^{+}
(\hat{\boldsymbol{\beta}}^{RE}-\hat{\boldsymbol{\beta}}^{UN})
\end{equation}
where $u^+= max \{0, u\}$ (\cite{Kibria}).

\section{Properties of Suggested Estimators}\label{sec4}
In this section, the aim is to provide the asymptotic properties of estimators in the previous section. To explore the properties when the prior information $\boldsymbol{H}\boldsymbol{\beta}=\boldsymbol{h}$ is wrong, we consider the sequence
of local alternatives:
\begin{equation}\label{Hn}
H_{n0}: ~ \boldsymbol{H}\boldsymbol{\beta}=\boldsymbol{h}+\dfrac{\boldsymbol{\gamma}}{\sqrt{n}}
\end{equation}
where $ \boldsymbol{\gamma} \in \mathbb{R}^r $ is a fixed vector of length $r$. It is obvious that the null hypothesis in \eqref{H0} is a special case of \eqref{Hn} and the value of $ \dfrac{\boldsymbol{\gamma}}{\sqrt{n}} $ shows the distance between the value of   $\boldsymbol{H}\boldsymbol{\beta}$ in \eqref{Hn} and \eqref{H0}.
\\
Under null-hypothesis in \eqref{Hn}, we define the asymptotic distributional bias of the estimator $\hat{\boldsymbol{\beta}}$ is defined as follows:
\begin{equation}\label{bias}
\mathbb{B}(\hat{\boldsymbol{\beta}})= \lim\limits_{n \to \infty} \mathbb{E} \big[ \sqrt{n}(\hat{\boldsymbol{\beta}}-\boldsymbol{\beta}) \big]
\end{equation}
and the asymptotic distributional mean squared error (AMSE) of the estimator $\hat{\boldsymbol{\beta}}$ is defined as:
\begin{equation}\label{cov}
AMSE(\hat{\boldsymbol{\beta}})= \lim\limits_{n \to \infty} \mathbb{E} \big[ \sqrt{n}(\hat{\boldsymbol{\beta}}-\boldsymbol{\beta}) \sqrt{n}(\hat{\boldsymbol{\beta}}-\boldsymbol{\beta})^T\big]
\end{equation}

We present the following lemma, which is proved by Judge and Bock \cite{judge}, that helps us to derive the properties of the suggested estimators.
\begin{Lemma}\label{lem1}
Let $Z_1= \sqrt{n}(\hat{\boldsymbol{\beta}}^{UN}-\boldsymbol{\beta})$, $Z_2= \sqrt{n}(\hat{\boldsymbol{\beta}}^{RE}-\boldsymbol{\beta})$ and $Z_3= \sqrt{n}(\hat{\boldsymbol{\beta}}^{UN}-\hat{\boldsymbol{\beta}}^{RE})$. Under null-hypothesis in \eqref{Hn} and regularity conditions of MLE, when $n$ increases:
\begin{equation}\label{need1}
 \left[ \begin{gathered}
  {Z_1} \hfill \\
  {Z_2} \hfill \\
  {Z_3} \hfill \\
\end{gathered}  \right] \sim N \left( {\left[ \begin{gathered}
  \pmb{0} \hfill \\
  -\boldsymbol{\kappa}\boldsymbol{\gamma}\hfill \\
  \boldsymbol{\kappa}\boldsymbol{\gamma} \hfill \\
\end{gathered}  \right] ,\left[ {\begin{array}{*{20}{c}}
 \boldsymbol{F}^{-1} & \boldsymbol{F}^{-1}- \boldsymbol{\kappa}_0 &  \boldsymbol{\kappa}_0 \\
 &  \boldsymbol{F}^{-1}- \boldsymbol{\kappa}_0 & \pmb{0} \\
 & &  \boldsymbol{\kappa}_0
\end{array}}\right]}\right)
\end{equation}
where $\boldsymbol{\kappa}_0= \boldsymbol{F}^{-1} \boldsymbol{H}^T (\boldsymbol{H} \boldsymbol{F}^{-1}\boldsymbol{H}^T)^{-1} \boldsymbol{H}\boldsymbol{F}^{-1}$.
\end{Lemma}

Based on Lemma \ref{lem1}, we provide the asymptotic properties of the suggested estimators in the subsequent theorems.

\begin{Theorem}\label{thm1}
Under null-hypothesis in \eqref{Hn} and regularity conditions, the asymptotic distributional bias of the suggested estimators are:
\begin{align*}
\mathbb{B}(\hat{\boldsymbol{\beta}}^{RE}) &= -\boldsymbol{\kappa \gamma}\\
\mathbb{B}(\hat{\boldsymbol{\beta}}^{JSE}) &= -(r-2) \boldsymbol{\kappa \gamma} \mathbb{E}(\chi^{-2}_r (\delta))\\
\mathbb{B}(\hat{\boldsymbol{\beta}}^{PJSE}) &= \mathbb{B}(\hat{\boldsymbol{\beta}}^{JSE})+\boldsymbol{\kappa \gamma} \Upsilon_r (r-2; \delta) \\
\mathbb{B}(\hat{\boldsymbol{\beta}}^{PTE}) &= -\boldsymbol{\kappa \gamma}\Upsilon_r \big(\chi^2_{(r, \alpha)}; \delta\big)
\end{align*}
where $\boldsymbol{\kappa}=  \boldsymbol{F}^{-1} \boldsymbol{H}^T (\boldsymbol{H} \boldsymbol{F}^{-1}\boldsymbol{H}^T)^{-1}$, $ \Upsilon_r \big(.; \delta\big) $ and $\mathbb{E}(\chi^{2k}_r (\delta))$ denote respectively for the cumulative distribution function and $k$th order moment of a non-central $\chi^2$ distribution with $r$ degrees of freedom and non-central parameter $\delta= \boldsymbol{\gamma}^T(\boldsymbol{H} \boldsymbol{F}^{-1}\boldsymbol{H}^T)^{-1}\boldsymbol{\gamma}$.
\end{Theorem}
\begin{proof}
See Appendix \ref{proof1}.
\end{proof}
For the asymptotic distributional  mean squared error of the estimators, we have the following Theorem.
\begin{Theorem}\label{thm2}
Under null-hypothesis in \eqref{Hn} and regularity conditions, the asymptotic distributional  mean squared error of the suggested estimators are:
\begin{align*}
AMSE(\hat{\boldsymbol{\beta}}^{UN}) &= \boldsymbol{F}^{-1}\\
AMSE(\hat{\boldsymbol{\beta}}^{RE}) &= \boldsymbol{F}^{-1}- \boldsymbol{\kappa}_0 + \boldsymbol{\kappa \gamma}\boldsymbol{\kappa}^T\boldsymbol{\gamma}^T\\
AMSE(\hat{\boldsymbol{\beta}}^{JSE}) &= \boldsymbol{F}+2(r-2)\boldsymbol{\kappa \gamma} \boldsymbol{\kappa}^T\boldsymbol{\gamma}^T \mathbb{E}\bigg[ \chi^{-2}_{r+2} (\delta) \bigg]+ (r-2) \boldsymbol{\kappa}_0  \mathbb{E}\bigg[\chi^{-4}_{r+2} (\delta)\bigg]  -2(r-2) \boldsymbol{\kappa}_0 \mathbb{E}\bigg[\chi^{-2}_{r+2} (\delta)\bigg]\\
& \qquad + (r-2)^2 \boldsymbol{\kappa \gamma}\boldsymbol{\kappa}^T\boldsymbol{\gamma}^T   \mathbb{E}\bigg[ \chi^{-4}_{r+4} (\delta) \bigg]-2 (r-2) \boldsymbol{\kappa \gamma}\boldsymbol{\kappa}^T\boldsymbol{\gamma}^T  \mathbb{E}\bigg[ \chi^{-2}_{r+4} (\delta) \bigg] \\
AMSE(\hat{\boldsymbol{\beta}}^{PJSE}) &=AMSE(\hat{\boldsymbol{\beta}}^{JSE})-2\boldsymbol{\kappa \gamma}\boldsymbol{\kappa}^T\boldsymbol{\gamma}^T \bigg[ \Upsilon_{r+2}(r-2; \lambda)- \Upsilon_{r+4}(r-2; \delta) \bigg] \\
& \qquad + 3 \boldsymbol{\kappa}_0  \Upsilon_{r+2} (r-2; \delta) -2 (r-2)  \boldsymbol{\kappa \gamma}\boldsymbol{\kappa}^T\boldsymbol{\gamma}^T \mathbb{E}\bigg[\chi^{-2}_{r+4} (\delta)I_{(\chi^{-2}_{r+4} (\delta) < r-2)}\bigg] \\
& \qquad -2(r-2) \boldsymbol{\kappa}_0 \mathbb{E}\bigg[\chi^{-2}_{r+2} (\delta)I_{(\chi^{-2}_{2+2} (\delta) < r-2)}\bigg] +2\boldsymbol{\kappa \gamma}\boldsymbol{\kappa}^T\boldsymbol{\gamma}^T  \Upsilon_{r+4} (r-2; \delta)\\
AMSE(\hat{\boldsymbol{\beta}}^{PTE}) &= \boldsymbol{F}^{-1}+ \boldsymbol{\kappa\gamma}\boldsymbol{\kappa}^T\boldsymbol{\gamma}^T \big[\Upsilon_{r+4}\big(\chi^2_{(r, \alpha)}; \delta\big)-\Upsilon_{r+4}\big(\chi^2_{(r, \alpha)}; \delta\big) \big] - \boldsymbol{\kappa}_0 \Upsilon_{r+2} \big(\chi^2_{(r, \alpha)}; \delta\big)
\end{align*}
\end{Theorem}
\begin{proof}
See Appendix \ref{proof2}.
\end{proof}

\section{Simulation Study}\label{sec5}
In this section, we provide a simulation study to compare the performance of the suggested estimators in the previous section.
The criterion used for thus comparison is the simulated relative efficiency (SRE) which is defined as:
\begin{equation}\label{SRE}
SRE(\hat{\boldsymbol{\beta}})= \dfrac{SMSE(\hat{\boldsymbol{\beta}}^{UN})} {SMSE(\hat{\boldsymbol{\beta}})}
\end{equation}
where SMSE is abbreviation of simulated mean squared error and $ \hat{\beta} $ is one of the suggested estimators in this paper.

The $n$ observations of covariates are generated from a standard multivariate distribution, e. g. $N_p (\boldsymbol{0}, \boldsymbol{I}_p)$ where the number of covariates and observation are considered as $p= 3, 6, 12$ and $n=50, 100, 200$.
The observations of the response variable are generated from Bell distribution with the parameter $W(\mu_i)$ where
\begin{equation}
\mu_i= \exp \{\beta_0+\beta_1 x_{i1}+ \beta_2 x_{i2}+ \ldots + \beta_p x_{ip} \}
\end{equation}
The true coefficients are taken to be $\boldsymbol{\beta}= (0, 1, 1, \ldots, 1)$.
In this simulation study, our objective is to investigate the impact of the deviation from the true coefficient vector.Therefore,  the sub-model in \eqref{H0} is defined by the following equations:
\begin{align}
&\beta_0=\tau \label{rest1}\\
&\beta_i - \beta_{i+1}=0 \label{rest2} \qquad i= 1, 2, \ldots, p-1
\end{align}
which lead to the following matrix form:
\begin{equation}
\begin{pmatrix}
1 & 0& 0& 0& \cdots & 0 & 0\\
0 & 1& -1& 0& \cdots & 0 & 0\\
0 & 0& 1& -1& \cdots & 0 & 0\\
\vdots & \vdots & \vdots & \vdots & \ddots  & \vdots & \vdots\\
0 & 0& 0& 0& \cdots & 1 & -1\\
\end{pmatrix}
\begin{pmatrix}
\beta_0\\
\beta_1\\
\beta_2\\
\vdots\\
\beta_p\\
\end{pmatrix}
=
\begin{pmatrix}
\tau\\
0\\
0\\
\vdots\\
0
\end{pmatrix}
\end{equation}
Here the value of the $\tau$ reflects the departure value from the real value of restrictions.
Thus, $\tau=0$ means the real value of $\beta_0=0$ but as the $\tau$ increases, we get away from the real value of $\beta_0$. Here, we consider $\tau \in [0,1]$.

To calculate the value of SMSE in \eqref{SRE}, we replicate the designed experiment $1000$ times and compute it by:
\begin{equation}\label{SMSE}
SMSE(\hat{\boldsymbol{\beta}})= \dfrac{1}{1000} \sum_{r=1}^{1000} (\hat{\boldsymbol{\beta}}_r-\boldsymbol{\beta})^T(\hat{\boldsymbol{\beta}}_r-\boldsymbol{\beta})
\end{equation}
where $ \hat{\boldsymbol{\beta}}_r $ is the estimated value of $ \boldsymbol{\beta} $ at the $r$th replication.

Since we used the simulated relative efficiency, it means that the value of SRE larger than one indicates that $\hat{\boldsymbol{\beta}}$ performs better than $\hat{\boldsymbol{\beta}}^{UN}$.

The results of the simulation are reported in Table \ref{MSE:ex} and Figure \ref{fig1} . The following conclusions can be provided from the table:
\begin{itemize}
\item When $\tau=0$, the performance of the RE is best in all situations.  However, when the value of $\tau$ increases, the SRE of the RE sharply decreases such that it even becomes inefficient.
\item When $\tau=0$, the PJSE performs better than JSE but when $\tau$ gets away from zero, the performance of both estimators becomes the same.
\item The PTE is better than James-Stein type estimators when the $\tau=0$.
\item when $\tau=0$, by increasing the number of covariates, the SRE increases.
\item Generally, when the sample size increases, the SRE of all estimators decreases.
\end{itemize}

\begin{figure}
\centering
\includegraphics[width=14cm,keepaspectratio]{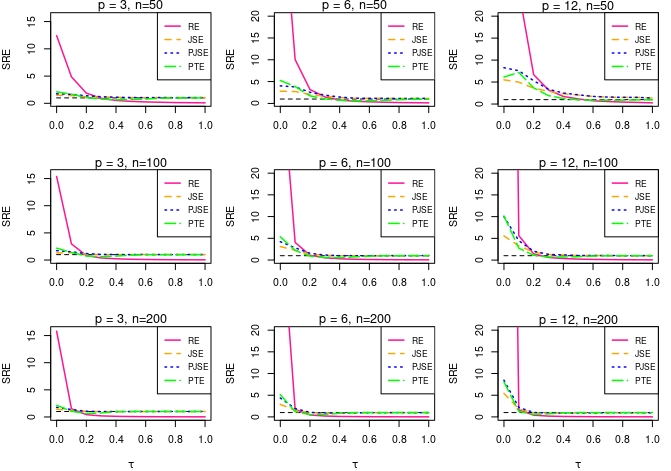}
\caption{The SRE of suggested estimators when $\tau \in [0, 1]$.}\label{fig1}
\end{figure}

\begin{landscape}
\begin{table}[!ht]
  \begin{center}
    \caption{The SRE of suggested estimators.}\label{MSE:ex}
 \footnotesize
\begin{tabular}{ccccccccccccccccc}
\toprule
        &   & \multicolumn{4}{c}{$p=3$}  && \multicolumn{4}{c}{$p=6$}  && \multicolumn{4}{c}{$p=12$} \\
        \cmidrule{3-6} \cmidrule{8-11} \cmidrule{13-16}
$n$ & $\tau$ &\text{RE} & \text{JSE}& \text{PJSE} & \text{PTE} &&  \text{RE} & \text{JSE}& \text{PJSE} & \text{PTE} &&  \text{RE} & \text{JSE}& \text{PJSE} & \text{PTE} \\
   \hline
50 & 0.0 &12.4228 &1.5033 &1.7488 &2.1112&&52.8411 &2.8042 &4.0199 &5.2595&&370.7309 &5.5062  &8.2652  &6.1435\\
    & 0.1 &4.8533 &1.4199 &1.6550 &1.6736&&10.0505 &2.7287 &3.7891 &3.8616&&25.8311 &4.9809  &7.5768  &7.1778\\
    & 0.2 &1.8652 &1.2032 &1.3908 &1.0606&&3.1735 &2.0691 &2.5510 &1.7379&&6.6886 &3.6138  &5.3081  &3.6787\\
    & 0.3 &0.9301 &1.1485 &1.2147 &0.7762&&1.4954 &1.6224 &1.9020 &1.0425&&3.2125 &2.8814  &3.3676  &1.9101\\
    & 0.4 &0.5249 &1.0989 &1.1139 &0.6520&&0.7642 &1.3463 &1.4600 &0.7175&&1.7167 &2.3471  &2.5025  &1.2587\\
    & 0.5 &0.3587 &1.0509 &1.0607 &0.6919&&0.5018 &1.1651 &1.2027 &0.5859&&1.1632 &2.0420  &2.1023  &0.9657\\
    & 0.6 &0.2570 &1.0387& 1.0395 &0.7722&&0.3642 &1.1575 &1.1665 &0.6122&&0.7327 &1.7205  &1.7331  &0.8564\\
    & 0.7 &0.1946 &1.0612 &1.0612 &0.9331&&0.2714 &1.1574 &1.1596 &0.7093&&0.5701 &1.5768  &1.5844  &0.8663\\
    & 0.8 &0.1373 &1.0446 &1.0446 &0.9732&&0.1952 &1.1257 &1.1257 &0.8596&&0.4382 &1.5252 & 1.5252  &0.8749\\
    & 0.9 &0.1138 &1.0583 &1.0583 &1.0000&&0.1587 &1.1389 &1.1389 &0.9591&&0.3591 &1.5130  &1.5130  &0.9384\\
    & 1.0 &0.0887 &1.0528 &1.0528 &1.0000&&0.1308 &1.1529 &1.1529 &0.9809&&0.2683 &1.3575  &1.3575  &0.9617\\
[5pt]
100 & 0.0 &15.3827 &1.3129 &1.7326 &2.1887&&46.4238 &3.0951 &4.2145 &5.3483&&287.1928 &5.5478 &10.0025 &10.1418\\
    & 0.1 &2.9689 &1.3566 &1.5177 &1.3964&&4.0473 &2.2777 &2.8153 &2.2206&&5.6220 &3.2858  &4.4510  &2.7358\\
    & 0.2 &0.8263 &1.1088 &1.1978 &0.7381&&1.0957 &1.3635 &1.5558 &0.8978&&1.4263 &1.8247  &2.0409  &1.1128\\
    & 0.3 &0.3903 &1.0057 &1.0348 &0.6325&&0.5166 &1.0901 &1.1461 &0.5871&&0.5941 &1.2645  &1.3110  &0.6970\\
    & 0.4 &0.2336 &1.0168 &1.0186 &0.7294&&0.2850 &0.9891 &0.9977 &0.5428&&0.3563 &1.1812  &1.1873  &0.7068\\
    & 0.5 &0.1431 &1.0007 &1.0007 &0.8982&&0.1797 &0.9989 &0.9989 &0.7181&&0.2199 &1.0406  &1.0415 & 0.7645\\
    & 0.6 &0.1056 &1.0281 &1.0281 &0.9882&&0.1223 &0.9618 &0.9618 &0.8796&&0.1523 &1.0569  &1.0569 & 0.9263\\
    & 0.7 &0.0760 &1.0070 &1.0070 &1.0000&&0.0938 &1.0290 &1.0290 &0.9761&&0.1174 &1.0597  &1.0597 & 0.9818\\
    & 0.8 &0.0576 &1.0166 &1.0166 &1.0000&&0.0708 &1.0071 &1.0071 &0.9825&&0.0866 &1.0180  &1.0180  &1.0000\\
    & 0.9 &0.0464 &1.0138 &1.0138 &1.0000&&0.0557 &1.0215 &1.0215 &1.0000&&0.0657 &1.0222  &1.0222 & 1.0000\\
    & 1.0 &0.0377 &1.0137 &1.0137 &1.0000&&0.0441 &1.0232 &1.0232 &1.0000&&0.0527 &1.0367  &1.0367  &1.0000 \\
    [5pt]
200 & 0.0 &15.7941 &1.4145 &1.7699 &2.1467&&48.3578 &2.8900 &4.3608 &5.1290&&255.0862 &5.4148 & 8.4551  &8.0564\\
    & 0.1 &1.5115 &1.2146 &1.3641 &1.0493&&1.7315 &1.5384 &1.9358 &1.2318&&1.6337 &1.8181  &2.1917  &1.2585\\
    & 0.2 &0.4170 &1.0304 &1.0441 &0.6273&&0.4516 &0.9773 &1.0433 &0.5496&&0.4001 &0.9824  &1.0181  &0.5894\\
    & 0.3 &0.1905 &0.9829 &0.9829 &0.7321&&0.2076 &0.9023 &0.9134 &0.5798&&0.1831 &0.9213  &0.9237  &0.7043\\
    & 0.4 &0.1115 &0.9992 &0.9995 &0.9621&&0.1085 &0.8724 &0.8724 &0.8384&&0.0997 &0.8904  &0.8904  &0.9415\\
    & 0.5 &0.0691 &0.9924 &0.9924 &1.0000&&0.0709 &0.9300 &0.9300 &1.0000&&0.0655 &0.9133  &0.9133  &1.0000\\
    & 0.6 &0.0490 &1.0042 &1.0042 &1.0000&&0.0486 &0.9951 &0.9951 &1.0000&&0.0435 &0.9214  &0.9214  &1.0000\\
    & 0.7 &0.0359 &1.0046 &1.0046 &1.0000&&0.0380 &0.9741 &0.9741 &1.0000&&0.0328 &0.9746  &0.9746  &1.0000\\
    & 0.8 &0.0276 &1.0070 &1.0070 &1.0000&&0.0278 &1.0019 &1.0019 &1.0000&&0.0245 &0.9617  &0.9617  &1.0000\\
    & 0.9 &0.0216 &1.0039 &1.0039 &1.0000&&0.0215 &0.9761 &0.9761 &1.0000&&0.0240 &0.9694  &0.9694  &1.0000\\
    & 1.0 &0.0171 &1.0105 &1.0105 &1.0000&&0.0183 &0.9994 &0.9994 &1.0000&&0.0152 &0.9497  &0.9497  &1.0000\\
\bottomrule
\end{tabular}
\end{center}
\end{table}
\end{landscape}

\section{Application to a real dataset}\label{sec6}
\cite{Myers} first used the dataset of Mine fracture to model count data. Nevertheless, since the variance is roughly 135 times greater than the mean (135.614), we were able to calculate the ratio of the distribution's variance to the mean and found that the PRM does not match the data. Since over-dispersion is present, the Bell regression is a better fit for this model.
 
To analyze the dataset, we utilized the bellreg package in the R programming language (\cite{bellreg}). The AIC value for the full model, considering all covariates, was $-1674.772$. The analysis revealed that variables $x_{1}$ and $x_{3}$ were not statistically significant. Consequently, we opted for a restricted model that only included variables $x_2$ and $x_{4}$. The AIC for the restricted model was -1675.556, indicating its superiority over the full model. Thus, we imposed restrictions on the model, specifically setting $\beta_1=0$ and $\beta_3=0$.
\\
The significance level for the pretest estimator was set at $0.05$. We employed the suggested estimators outlined in this paper to analyze the data. The estimated parameter values are presented in Table \ref{real1}. To assess the performance of the suggested estimators, we applied the bootstrap technique with a sample size of $n=40$ and $1000$ replications. Subsequently, we computed the bootstrapped relative efficiencies (BRE) for each bootstrap sample, with the results reported in the last row of Table \ref{real1}.
\\
The findings indicate that, in terms of BRE and standard error, the RE, positive James-Stein estimator, and pretest estimator outperformed the UN. However, for this particular dataset, the James-Stein estimator exhibited inefficiency. Among the suggested estimators, the RE demonstrated a higher BRE value (2.7154), indicating greater efficiency, while the James-Stein estimator had a lower BRE value (0.1750), denoting inefficiency.

\begin{table}[!ht]
  \begin{center}
    \caption{Coefficients, standard error (in the brackets) and BRE of suggested estimators.}\label{real1}
 \footnotesize
\begin{tabular}{cccccc}
\toprule
Covariates & \text{UN} &\text{RE} & \text{JSE}& \text{PJSE} & \text{PTE}   \\
\toprule
$x_1$ &-0.8087 (0.0154)&  0.0000  (0.0000)& 0.4774 (0.0395) & -0.0894 (0.0068)&  0.0013 (0.0013)\\
$x_2$ &1.7047  (0.0158)& 0.8629  (0.0021)& 0.3661 (0.0416) & 0.9564  (0.0071)& 0.8620 (0.0015)\\
$x_3$ &-13.1547 (0.5967)&  0.0000 (0.0000)&  8.0152 (1.3019) &-1.3660 (0.1869 )& -0.1101 (0.1101)\\
$x_4$ &0.0917  (0.0007)& 0.0728  (0.0003)& 0.0613 (0.0016) & 0.0748 (0.0003)&  0.0729 (0.0002)\\
\toprule
BRE & -& 2.7154 & 0.1750 & 2.5531 & 2.4947\\
\bottomrule
\end{tabular}
\end{center}
\end{table}

\section{Conclusion}\label{sec7}
In this paper, different types of shrinkage estimators, including the pretest, James-Stein, and positive James-Stein estimators, were applied based on both unrestricted and restricted estimators for the BRM. The asymptotic bias and mean squared error of the suggested estimators were derived under the local alternative hypothesis.
Furthermore, a simulation study was conducted to evaluate the performance of the suggested estimators. The results demonstrate that when the prior information on the coefficients is accurate, the suggested estimators exhibit significantly higher relative efficiencies compared to the UN.
To further validate our findings, the suggested estimators were applied to a real-world dataset. The results reaffirmed the superiority of the suggested estimators, as they consistently displayed high overall relative efficiencies.


\newpage

\appendix
\section{Proof of Theorem \ref{thm1}}\label{proof1}
For obtain the asymptotic distribution bias of the RE, by using Lemma \ref{lem1}, we have:
\begin{equation}
\mathbb{B}(\hat{\pmb{\beta}}^{RE})= \mathbb{E}\big(\sqrt{n}(\hat{\boldsymbol{\beta}}^{RE}-\boldsymbol{\beta})\big)= \mathbb{E}(Z_2)= - \boldsymbol{\kappa} \boldsymbol{\gamma}.
\end{equation}
For obtain the asymptotic distribution bias of the JSE, we have:
\begin{align}
\mathbb{B}\big(\hat{\boldsymbol{\beta}}^{JSE}\big)& =\mathbb{E}\big(\sqrt{n}(\hat{\boldsymbol{\beta}}^{JSE}-\boldsymbol{\beta})\big) \nonumber \\
&= \mathbb{E} \bigg[\sqrt{n} \bigg\{(\hat{\boldsymbol{\beta}}^{RE}- \boldsymbol{\beta})+ \big\{1-\dfrac{r-2}{F_n}\big\}\big(\hat{\boldsymbol{\beta}}^{UE}-\hat{\boldsymbol{\beta}}^{RE} \big) \bigg\} \bigg] \nonumber\\
& =\mathbb{E} \bigg[ Z_2 +  \big\{1-\dfrac{r-2}{F_n}\big\} Z_3 \bigg] \nonumber \\
& = \mathbb{E} \big(Z_2  \big) + \mathbb{E}\bigg[  \big\{1-\dfrac{r-2}{F_n}\big\} Z_3 \bigg] \nonumber
\end{align}
Following Lemma \ref{lem1}, we have:
\begin{align}
\mathbb{E}\big(\sqrt{n}(\hat{\boldsymbol{\beta}}^{JSE}-\boldsymbol{\beta})\big) &= -\boldsymbol{\kappa} \boldsymbol{\gamma} + \boldsymbol{\kappa} \boldsymbol{\gamma} \mathbb{E}\bigg[ 1-(r-2) \chi^{-2}_{r+2} (\delta) \bigg]  \nonumber \\
& = -(r-2) \boldsymbol{\kappa} \boldsymbol{\gamma} \mathbb{E}\bigg[ \chi^{-2}_{r+2} (\delta) \bigg]
\end{align}
To get the asymptotic distributional bias of the PJSE, we can write:
\begin{align}
\mathbb{B} \big(\hat{\boldsymbol{\beta}}^{PJSE}\big)& =\mathbb{E}\big(\sqrt{n}(\hat{\boldsymbol{\beta}}^{PJSE}-\boldsymbol{\beta})\big) \nonumber \\
&= \mathbb{E} \bigg[\sqrt{n} \bigg\{\hat{\boldsymbol{\beta}}^{RE}+\big\{1-\dfrac{r-2}{F_n}\big\}^{+}\big(\hat{\boldsymbol{\beta}}^{UN}-\hat{\boldsymbol{\beta}}^{RE}\big)- \boldsymbol{\beta} \bigg\} \bigg] \nonumber\\
& = \mathbb{E} \bigg[\sqrt{n} \bigg\{(\hat{\boldsymbol{\beta}}^{JSE}- \boldsymbol{\beta})+ \big(\hat{\boldsymbol{\beta}}^{UN}-\hat{\boldsymbol{\beta}}^{RE} \big) I_{(F_n < (r-2))} \bigg] \nonumber\\
& =  \mathbb{E} \bigg[\sqrt{n} (\hat{\boldsymbol{\beta}}^{JSE}- \boldsymbol{\beta})+ \sqrt{n}\big(\hat{\boldsymbol{\beta}}^{UN}-\hat{\boldsymbol{\beta}}^{RE} \big) I_{(F_n< (r-2))} \bigg] \nonumber\\
& =\mathbb{E} \bigg[\sqrt{n} (\hat{\boldsymbol{\beta}}^{JSE}- \boldsymbol{\beta})+ Z_3 I_{(F_n<(r-2))} \bigg] \nonumber\\
& = \mathbb{B}\big(\hat{\boldsymbol{\beta}}^{JSE}\big) + \mathbb{E}\bigg[I_{(F_n< (r-2))} Z_3 \bigg] \nonumber \\
& = \mathbb{B} \big(\hat{\boldsymbol{\beta}}^{JSE}\big) + \mathbb{E} (Z_3) \mathbb{E}\bigg[I_{(F_n< (r-2))}\bigg] \nonumber \\
& =  \mathbb{B} \big(\hat{\boldsymbol{\beta}}^{JSE}\big) + \boldsymbol{\kappa} \boldsymbol{\gamma}\Upsilon_{r+2} (r-2; \delta) \nonumber \\
& = -(r-2) \boldsymbol{\kappa} \boldsymbol{\gamma} \mathbb{E}\bigg[ \chi^{-2}_{r+2} (\delta) \bigg]+ \boldsymbol{\kappa} \boldsymbol{\gamma}\Upsilon_{r+2} (r-2; \delta) \nonumber.
\end{align}
Finally, for the asymptotic distributional bias of the pretest estimator, we have:
\begin{align}
\mathbb{B}\big(\hat{\boldsymbol{\beta}}^{PTE}\big)& =\mathbb{E}\big(\sqrt{n}(\hat{\boldsymbol{\beta}}^{PTE}-\boldsymbol{\beta})\big) \nonumber \\
& = \mathbb{E} \bigg[\sqrt{n} \bigg\{\hat{ (\boldsymbol{\beta}}^{UN}- \boldsymbol{\beta}) - (\hat{\boldsymbol{\beta}}^{UN}-\hat{\boldsymbol{\beta}}^{RE}) I_{(F_n \leq \chi^2_{(r, \alpha)})} \bigg\} \bigg] \nonumber\\
& =\mathbb{E} \bigg[Z_1- Z_3 I_{(F_n \leq \chi^2_{(r, \alpha)})} \bigg] \nonumber\\
& = \mathbb{E}  (Z_1)- \mathbb{E} \bigg[Z_3 I_{(F_n \leq \chi^2_{(r, \alpha)})} \bigg] \nonumber\\
& =- \mathbb{E}  (Z_3)\mathbb{E} \bigg[I_{(F_n \leq \chi^2_{(r, \alpha)})} \bigg] \nonumber\\
& = - \boldsymbol{\kappa} \boldsymbol{\gamma}\Upsilon_{r+2}(\chi^2_{(r, \alpha)}; \delta)\nonumber
\end{align}

\newpage
\section{Proof of Theorem \ref{thm2}}\label{proof2}
To obtain the AMSE of UN, we use \eqref{cov} and Lemma \ref{lem1}. Therefore:
\begin{align}
AMSE(\hat{\boldsymbol{\beta}}^{UN}) &=\mathbb{E}\bigg[ \sqrt{n}(\hat{\boldsymbol{\beta}}^{UN}-\boldsymbol{\beta})\sqrt{n}(\hat{\boldsymbol{\beta}}^{UN}-\boldsymbol{\beta})^T \bigg] \nonumber\\
&=  \mathbb{E}\bigg[ Z_1 Z_1^T \bigg] = \boldsymbol{F}^{-1}
\end{align}
For the AMSE of RE, we have:
\begin{align}
AMSE(\hat{\boldsymbol{\beta}}^{RE}) &=\mathbb{E}\bigg[ \sqrt{n}(\hat{\boldsymbol{\beta}}^{RE}-\boldsymbol{\beta})\sqrt{n}(\hat{\boldsymbol{\beta}}^{RE}-\boldsymbol{\beta})^T \bigg] \nonumber\\
& =  \mathbb{E}\bigg[ Z_2 Z_2^T \bigg] \nonumber\\
& =  cov(Z_2)+  \bigg[\mathbb{E}(Z_2)\bigg]\bigg[\mathbb{E}(Z_2)\bigg]^T\nonumber \\
& =   \big[ \boldsymbol{F}^{-1} - \boldsymbol{\kappa}_0 \big] +\boldsymbol{\kappa} \boldsymbol{\gamma} \boldsymbol{\gamma} ^T \boldsymbol{\kappa}^T
\end{align}
For AMSE of James-Stein estimator, we have:
\begin{align} \label{e1}
AMSE(\hat{\pmb{\beta}}^{JSE}) &=\mathbb{E}\bigg[ \sqrt{n}(\hat{\pmb{\beta}}^{JSE}-\pmb{\beta})\sqrt{n}(\hat{\pmb{\beta}}^{JSE}-\pmb{\beta})^T \bigg] \nonumber\\
& =  \mathbb{E}\bigg[ \bigg(Z_2+ \big\{1-\dfrac{r-2}{F_n}\big\} Z_3\bigg) \bigg(Z_2+ \big\{1-\dfrac{r-2}{F_n}\big\} Z_3\bigg)^T\bigg] \nonumber\\
& =  \mathbb{E}\bigg[Z_2 Z_2^T\bigg] + \mathbb{E}\bigg[Z_2  \big\{1-\dfrac{r-2}{F_n}\big\} Z_3^T\bigg]+ \mathbb{E}\bigg[Z_2^T \big\{1-\dfrac{r-2}{F_n}\big\} Z_3\bigg]  \nonumber\\
& \qquad + \mathbb{E}\bigg[ \big\{1-\dfrac{r-2}{F_n}\big\}^2 Z_3Z_3^T\bigg]\nonumber \\
& = cov(Z_2)+  \bigg[\mathbb{E}(Z_2)\bigg]\bigg[\mathbb{E}(Z_2)\bigg]^T + 2 \underbrace{\mathbb{E}\bigg[Z_2  \big\{1-\dfrac{r-2}{F_n}\big\} Z_3^T\bigg]}_{T_1}  \nonumber\\
& \qquad + \underbrace{\mathbb{E}\bigg[ \big\{1-\dfrac{r-2}{F_n}\big\}^2 Z_3Z_3^T\bigg]}_{T_2}
\end{align}
Now, following Lemma \ref{lem1}:
\begin{align}
T_1 & =  \mathbb{E}\bigg[Z_2  \big\{1-\dfrac{r-2}{F_n}\big\} Z_3^T\bigg] \nonumber\\
& =\mathbb{E}\bigg[ \mathbb{E}\bigg[Z_2  \big\{1-\dfrac{r-2}{F_n}\big\} Z_3^T\vert Z_3\bigg]\bigg] \nonumber
\end{align}
\begin{align}\label{m1}
T_1& =\mathbb{E}\bigg[ \mathbb{E}\bigg[Z_2 \vert Z_3\bigg] \big\{1-\dfrac{r-2}{F_n}\big\} Z_3^T\bigg] \nonumber\\
& =\mathbb{E}\bigg[ \bigg\{ \mathbb{E}(Z_2)+ cov(Z_2, Z_3) [cov(Z_3)]^{-1} \big( Z_3-\mathbb{E}(Z_2) \big)\bigg\} \big\{1-\dfrac{r-2}{F_n}\big\} Z_3^T\bigg] \nonumber\\
& = \mathbb{E}\bigg[-\boldsymbol{\kappa} \boldsymbol{\gamma} \big\{1-\dfrac{r-2}{F_n}\big\} Z_3^T\bigg] \nonumber\\
& =  -\boldsymbol{\kappa} \boldsymbol{\gamma}  \mathbb{E}\bigg[\big\{1-\dfrac{r-2}{F_n}\big\} Z_3^T\bigg] \nonumber\\
& =  -\boldsymbol{\kappa} \boldsymbol{\gamma}  \mathbb{E} (Z_3^T) \mathbb{E} \bigg[1-(r-2) \chi^{-2}_{r+2} (\delta)\bigg] \nonumber\\
&=-\boldsymbol{\kappa} \boldsymbol{\gamma}  \boldsymbol{\gamma}^T \boldsymbol{\kappa}^T \mathbb{E} \bigg[1-(r-2) \chi^{-2}_{r+2} (\delta)\bigg] \nonumber\\
& =  -\boldsymbol{\kappa} \boldsymbol{\gamma} \boldsymbol{\gamma}^T \boldsymbol{\kappa}^T + (r-2) \boldsymbol{\kappa} \boldsymbol{\gamma} \boldsymbol{\gamma}^T \boldsymbol{\kappa}^T \mathbb{E} \bigg[\chi^{-2}_{r+2} (\delta)\bigg]
\end{align}
For $M_2$, we have:
 \begin{align} \label{m2}
T_2 & = \mathbb{E}\bigg[ \big\{1-\dfrac{r-2}{F_n}\big\}^2 Z_3Z_3^T\bigg] \nonumber\\
& = cov(Z_3) \mathbb{E}\bigg[ \{1-(r-2) \chi^{-2}_{r+2} (\delta) \}^2\bigg]+ \bigg[\mathbb{E}(Z_3)\bigg] \bigg[\mathbb{E}(Z_3)\bigg]^T \mathbb{E}\bigg[ \{1-(r-2) \chi^{-2}_{r+4} (\delta) \}^2\bigg] \nonumber\\
& = \boldsymbol{\kappa}_0 \mathbb{E}\bigg[ \{1-(r-2) \chi^{-2}_{r+2} (\delta) \}^2\bigg]+ \boldsymbol{\kappa} \boldsymbol{\gamma} \boldsymbol{\gamma}^T \boldsymbol{\kappa}^T \mathbb{E}\bigg[ \{1-(r-2) \chi^{-2}_{r+4} (\delta) \}^2\bigg]
\end{align}
Consequently,
\begin{align}
AMSE(\hat{\pmb{\beta}}^{JSE})
&=  \big[\boldsymbol{F}^{-1} - \boldsymbol{\kappa}_0 \big] + \boldsymbol{\kappa} \boldsymbol{\gamma} \boldsymbol{\gamma}^T \boldsymbol{\kappa}^T -2\boldsymbol{\kappa} \boldsymbol{\gamma} \boldsymbol{\gamma}^T \boldsymbol{\kappa}^T +2 (r-2) \boldsymbol{\kappa} \boldsymbol{\gamma} \boldsymbol{\gamma}^T \boldsymbol{\kappa}^T \mathbb{E} \bigg[\chi^{-2}_{r+2} (\delta)\bigg] \nonumber \\
& \qquad + \boldsymbol{\kappa}_0 \mathbb{E}\bigg[ \{1-(r-2)\chi^{-2}_{q+2} (\lambda) \}^2\bigg]+ \boldsymbol{\kappa} \boldsymbol{\gamma} \boldsymbol{\gamma}^T \boldsymbol{\kappa}^T \mathbb{E}\bigg[ \{1-(r-2) \chi^{-2}_{r+4} (\delta) \}^2\bigg] \nonumber \\
& = \boldsymbol{F}^{-1}+ 2 (r-2) \boldsymbol{\kappa} \boldsymbol{\gamma} \boldsymbol{\gamma}^T \boldsymbol{\kappa}^T \mathbb{E}\bigg[ \chi^{-2}_{r+2} (\delta) \bigg] + (r-2)^2 \boldsymbol{\kappa}_0 \mathbb{E}\bigg[\chi^{-4}_{r+2} (\delta)\bigg]  \nonumber \\
& \qquad  -2(r-2)\boldsymbol{\kappa}_0 \mathbb{E}\bigg[\chi^{-2}_{r+2} (\delta)\bigg] + (r-2)^2 \boldsymbol{\kappa} \boldsymbol{\gamma} \boldsymbol{\gamma}^T \mathbb{E}\bigg[ \chi^{-4}_{r+4} (\delta) \bigg] \nonumber \\
& \qquad -2 (r-2) \boldsymbol{\kappa} \boldsymbol{\gamma} \boldsymbol{\gamma}^T \boldsymbol{\kappa}^T  \mathbb{E}\bigg[ \chi^{-2}_{r+4} (\delta) \bigg]
\end{align}
For the AMSE of positive James-Stein estimator, we have:

\begin{align}\label{e4}
MSE(\hat{\boldsymbol{\beta}}^{PJSE}) &=\mathbb{E}\bigg[ \sqrt{n}(\hat{\boldsymbol{\beta}}^{PJSE}-\boldsymbol{\beta})\sqrt{n}(\hat{\boldsymbol{\beta}}^{PJSE}-\boldsymbol{\beta})^T \bigg] \nonumber\\
& = \mathbb{E}\bigg[\bigg\{\sqrt{n}(\hat{\boldsymbol{\beta}}^{JSE}-\boldsymbol{\beta}) + Z_3 I_{(F_n<(r-2))} \bigg\} \bigg\{\sqrt{n}(\hat{\boldsymbol{\beta}}^{JSE}-\boldsymbol{\beta}) + Z_3 I_{(F_n< (r-2))} \bigg\}^T \bigg]  \nonumber\\
& = \mathbb{E}\bigg[ \sqrt{n}(\hat{\boldsymbol{\beta}}^{JSE}- \boldsymbol{\beta})\sqrt{n}(\hat{\boldsymbol{\beta}}^{JSE}- \boldsymbol{\beta})^T\bigg] + \mathbb{E}\bigg[\sqrt{n}(\hat{\boldsymbol{\beta}}^{JSE}- \boldsymbol{\beta}) Z_3^T I_{(F_n< (r-2))}\bigg]  \nonumber\\
& \qquad + \mathbb{E}\bigg[\sqrt{n}(\hat{\boldsymbol{\beta}}^{JSE}- \boldsymbol{\beta})^T Z_3 I_{(F_n< (r-2))}\bigg]+ \mathbb{E}\bigg[Z_3 Z_3^T  I_{(F_n< (r-2))}\bigg] \nonumber\\
& = AMSE(\hat{\boldsymbol{\beta}}^{JSE}) + 2 \underbrace{ \mathbb{E}\bigg[\sqrt{n}(\hat{\boldsymbol{\beta}}^{JSE}- \boldsymbol{\beta}) Z_3^T I_{(F_n< (r-2))}\bigg]}_{T_3}+\underbrace{\mathbb{E}\bigg[Z_3 Z_3^T  I_{(F_n< (r-2))}\bigg]}_{T_4}
\end{align}
where
\begin{align}\label{m3}
T_3 &=  \mathbb{E}\bigg[\sqrt{n}(\hat{\boldsymbol{\beta}}^{JSE}- \boldsymbol{\beta}) Z_3^T I_{(F_n< (r-2))}\bigg] \nonumber\\
 & = \mathbb{E}\bigg[\bigg(Z_2+ \{1- (r-2)F_n^{-1}\} Z_3\bigg) Z_3^T I_{(F_n< (r-2))}\bigg] \nonumber\\
& = \mathbb{E}\bigg[Z_2 Z_3^T I_{(F_n< c)}+ Z_3 Z_3^T \{1- cF_n^{-1}\} I_{(F_n< (r-2))}\bigg]\nonumber \\
& = \mathbb{E}\bigg[ \mathbb{E}\bigg[Z_2 Z_3^T I_{(F_n< c)}\vert Z_3\bigg]\bigg]+  \mathbb{E}\bigg[Z_3 Z_3^T \{1- cF_n^{-1}\} I_{(F_n< c)}\bigg]\nonumber\\
& = \mathbb{E}\bigg[ \mathbb{E}\bigg[Z_2 | Z_3\bigg]Z_3^T I_{(F_n< c)}\bigg]+ \mathbb{E}\bigg[Z_3 Z_3^T \{1- cF_n^{-1}\} I_{(F_n< c)}\bigg]\nonumber\\
& =\mathbb{E}\bigg[  \bigg\{ \mathbb{E}(Z_2)+ cov(Z_2, Z_3) [cov(Z_3)]^{-1} \big( Z_3-\mathbb{E}(Z_2)) \bigg\} Z_3^T I_{(F_n< c)}\bigg]\nonumber\\
& \qquad + \mathbb{E}\bigg[Z_3 Z_3^T \{1- cF_n^{-1}\} I_{(F_n< (r-2))}\bigg]\nonumber\\
 & =\mathbb{E}\bigg[\mathbb{E}(Z_2) Z_3^T I_{(F_n< c)}\bigg]+ \mathbb{E}\bigg[Z_3 Z_3^T \{1- cF_n^{-1}\} I_{(F_n< c)}\bigg]\nonumber\\
 & = - \boldsymbol{\kappa} \boldsymbol{\gamma} \mathbb{E}\bigg[Z_3^T I_{(F_n < c)}\bigg]+ cov(Z_3)\mathbb{E}\bigg[  \{1- c\chi^{-2}_{r+2} (\delta) \} I_{(\chi^{-2}_{r+2} (\delta) < (r-2))}\bigg]\nonumber\\
& \qquad +\mathbb{E}(Z_3)\mathbb{E}(Z_3)^T \mathbb{E}\bigg[\{1- (r-2)\chi^{-2}_{r+4} (\delta) \} I_{(\chi^{-2}_{r+4} (\lambda) < c)}\bigg]\nonumber\\
& = - \boldsymbol{\kappa} \boldsymbol{\gamma} \boldsymbol{\gamma}^T \boldsymbol{\kappa}^T H_{q+2}((r-2); \delta) +  \boldsymbol{\kappa}_0 \mathbb{E}\bigg[ \{1- (r-2)\chi^{-2}_{r+2} (\lambda) \} I_{(\chi^{-2}_{r+2} (\delta) < (r-2))}\bigg]\nonumber\\
& \qquad +\boldsymbol{\kappa} \boldsymbol{\gamma} \boldsymbol{\gamma}^T \boldsymbol{\kappa}^T \mathbb{E}\bigg[\{1- c\chi^{-2}_{q+4} (\delta) \} I_{(\chi^{-2}_{r+4} (\delta) < (r-2))}\bigg]\nonumber
\end{align}
\begin{align}
T_3& = - \boldsymbol{\kappa} \boldsymbol{\gamma} \boldsymbol{\gamma}^T \boldsymbol{\kappa}^T \Upsilon_{r+2}((r-2); \delta) + \boldsymbol{\kappa}_0  \Upsilon_{r+2} ((r-2); \delta) - (r-2) \boldsymbol{\kappa}_0 \mathbb{E}\bigg[\chi^{-2}_{r+2} (\lambda)I_{(\chi^{-2}_{r+2} (\delta) < (r-2))}\bigg]
\nonumber\\
& \qquad +\boldsymbol{\kappa} \boldsymbol{\gamma} \boldsymbol{\gamma}^T \boldsymbol{\kappa}^T  \Upsilon_{r+4}((r-2); \lambda) - (r-2)  \boldsymbol{\kappa} \boldsymbol{\gamma} \boldsymbol{\gamma}^T \boldsymbol{\kappa}^T  \mathbb{E}\bigg[\chi^{-2}_{r+4} (\delta) I_{(\chi^{-2}_{r+4} (\delta) <(r-2))}\bigg]
\end{align}
and
\begin{align} \label{m4}
T_4 &=  \mathbb{E}\bigg[Z_3 Z_3^T  I_{(F_n< (r-2))}\bigg] \nonumber \\
& =  cov(Z_3)  \mathbb{E}\bigg[I_{(\chi^{-2}_{r+2} (\delta)< (r-2))}\bigg] + \mathbb{E}(Z_3)\mathbb{E}(Z_3)^T  \mathbb{E}\bigg[I_{(\chi^{-2}_{r+4} (\delta)< (r-2))}\bigg]  \nonumber \\
& =  \phi \boldsymbol{\kappa}_0  \mathbb{E}\bigg[I_{(\chi^{-2}_{r+2} (\delta)< (r-2))}\bigg] +  \boldsymbol{\kappa} \xi \xi^T \boldsymbol{\kappa}^T  \mathbb{E}\bigg[I_{(\chi^{-2}_{r+4} (\delta)< (r-2))}\bigg]  \nonumber \\
& = \boldsymbol{\kappa}_0   \Upsilon_{r+2}(r-2; \delta) +  \boldsymbol{\kappa} \boldsymbol{\gamma} \boldsymbol{\gamma}^T \boldsymbol{\kappa}^T   \Upsilon_{r+4}(r-2; \delta)
\end{align}
Replacing \eqref{m3} and \eqref{m4} in \eqref{e4}, we have:
\begin{align}\label{e5}
MSE(\hat{\pmb{\beta}}^{PJSE}) &=AMSE(\hat{\pmb{\beta}}^{JSE}) -2 \boldsymbol{\kappa} \boldsymbol{\gamma} \boldsymbol{\gamma}^T \boldsymbol{\kappa}^T \bigg[ \Upsilon_{r+2}(r-2; \delta)-\Upsilon_{r+4}(r-2; \delta) \bigg] \nonumber\\
& \qquad  + 3\boldsymbol{\kappa}_0  \Upsilon_{r+2} (r-2; \delta)  - 2(r-2)\boldsymbol{\kappa}_0 \mathbb{E}\bigg[\chi^{-2}_{r+2} (\delta)I_{(\chi^{-2}_{r+2} (\delta) < (r-2))}\bigg]  \nonumber\\
& \qquad -2(r-2)\boldsymbol{\kappa} \boldsymbol{\gamma} \boldsymbol{\gamma}^T \boldsymbol{\kappa}^T \mathbb{E}\bigg[\chi^{-2}_{r+4} (\delta)I_{(\chi^{-2}_{r+4} (\delta) < (r-2))}\bigg] \nonumber\\
& \qquad +2 \boldsymbol{\kappa} \boldsymbol{\gamma} \boldsymbol{\gamma}^T \boldsymbol{\kappa}^T  \Upsilon_{r+4} (r-2; \delta)
\end{align}
For pretest estimator, we have:
\begin{align}
AMSE(\hat{\boldsymbol{\beta}}^{PTE}) &=\mathbb{E}\bigg[ \sqrt{n}(\hat{\boldsymbol{\beta}}^{PTE}-\boldsymbol{\beta})\sqrt{n}(\hat{\boldsymbol{\beta}}^{PTE}-\boldsymbol{\beta})^T \bigg] \nonumber\\
& = \mathbb{E}\bigg[\big(Z_1- Z_3 I_{(F_n \leq \chi^2_{(r, \alpha)})}\big)\big(Z_1- Z_3 I_{(F_n \leq \chi^2_{(r, \alpha)})}\big)^T \bigg] \nonumber\\
& = \mathbb{E}\bigg[Z_1Z_1^T\bigg]- \mathbb{E}\bigg[Z_1Z_3^T I_{(F_n \leq \chi^2_{(r, \alpha)})}\bigg]- \mathbb{E}\bigg[Z_1^TZ_3 I_{(F_n \leq \chi^2_{(r, \alpha)})}\bigg] \nonumber\\
& \qquad + \mathbb{E}\bigg[Z_3Z_3^T I_{(F_n \leq \chi^2_{(r, \alpha)})}\bigg]  \nonumber\\
& = cov(Z_1)+ \bigg[\mathbb{E}(Z_1)\bigg] \bigg[\mathbb{E}(Z_1)\bigg]^T -2 \mathbb{E}\bigg[Z_1^TZ_3 I_{(F_n \leq \chi^2_{(r, \alpha)})}\bigg]\nonumber\\
& \qquad + \mathbb{E}\bigg[Z_3Z_3^T I_{(F_n \leq \chi^2_{(r, \alpha)})}\bigg]  \nonumber
\end{align}
\begin{align}
AMSE(\hat{\boldsymbol{\beta}}^{PTE})& = \boldsymbol{F}^{-1}  -2 \mathbb{E}\bigg[\mathbb{E}\bigg[Z_1Z_3^T I_{(F_n \leq \chi^2_{(r, \alpha)})}|Z_3 \bigg] \bigg]+  \mathbb{E}\bigg[Z_3Z_3^T I_{(F_n \leq \chi^2_{(r, \alpha)})}\bigg]  \nonumber\\
& = \boldsymbol{F}^{-1}  -2 \mathbb{E}\bigg[\mathbb{E}\bigg[Z_1|Z_3 \bigg] Z_3^T I_{(F_n \leq \chi^2_{(r, \alpha)})}\bigg]+  \mathbb{E}\bigg[Z_3Z_3^T I_{(F_n \leq \chi^2_{(r, \alpha)})}\bigg]  \nonumber\\
& = \boldsymbol{F}^{-1}  -2 \mathbb{E}\bigg[ \bigg\{\mathbb{E}(Z_1) + cov(Z_1, Z_3) \bigg[ cov(Z_3)\bigg]^{-1}(Z_3- \mathbb{E}(Z_3) )\bigg\} Z_3^T I_{(F_n \leq k)}\bigg] \nonumber\\
&\qquad +  \mathbb{E}\bigg[Z_3Z_3^T I_{(F_n \leq \chi^2_{(r, \alpha)})}\bigg]  \nonumber\\
& = \boldsymbol{F}^{-1}  -2   \mathbb{E}\bigg[\big(Z_3 - \boldsymbol{\kappa} \boldsymbol{\gamma} \big) Z_3^T I_{(F_n \leq \chi^2_{(r, \alpha)})}\bigg]  + \mathbb{E}\bigg[Z_3Z_3^T I_{(F_n \leq \chi^2_{(r, \alpha)})}\bigg]  \nonumber\\
 & = \boldsymbol{F}^{-1}  -\mathbb{E}\bigg[Z_3Z_3^T I_{(F_n \leq \chi^2_{(r, \alpha)})}\bigg]  +2\boldsymbol{\kappa} \boldsymbol{\gamma}   \mathbb{E}\bigg[Z_3^T I_{(F_n \leq \chi^2_{(r, \alpha)})}\bigg]  \nonumber\\
& = \boldsymbol{F}^{-1}  - cov(Z_3) \Upsilon_{r+2}(\chi^2_{(r, \alpha)}; \delta)- \bigg[\mathbb{E}(Z_1)\bigg]\bigg[\mathbb{E}(Z_1)\bigg]^T \Upsilon_{r+4}(\chi^2_{(r, \alpha)}; \delta) \nonumber\\
& \qquad + 2\boldsymbol{\kappa} \boldsymbol{\gamma} \boldsymbol{\gamma}^T \boldsymbol{\kappa}^T H_{r+2}(\chi^2_{(r, \alpha)}; \delta) \nonumber\\
& = \boldsymbol{F}^{-1} -  \boldsymbol{\kappa}_0 \Upsilon_{r+2}(\chi^2_{(r, \alpha)}; \delta) + \boldsymbol{\kappa} \boldsymbol{\gamma} \boldsymbol{\gamma}^T \boldsymbol{\kappa}^T\big(2 \Upsilon_{r+2}(\chi^2_{(r, \alpha)}; \delta)-\Upsilon_{r+4}(\chi^2_{(r, \alpha)}; \delta) \big)\nonumber
\end{align}

\end{document}